\documentclass[a4paper]{article}

\usepackage[latin1]{inputenc}
\usepackage[T1]{fontenc}
\usepackage[francais,english]{babel}

\usepackage{amsmath,amssymb,amsfonts,amsthm}
\usepackage{amsfonts}
\usepackage{amssymb}
\usepackage{latexsym}
\usepackage{comment}

\newcommand{\R}{\mathbb R}

\def\R{\mathbb R}

\newcommand{\be}{\begin{equation}}
\newcommand{\ee}{\end{equation}}
\def\lg{\langle}
\def\rg{\rangle}

\def\cb{B}

\def\ds{\displaystyle}

\newtheorem{Theorem}{Theorem}[section]

\newtheorem{Lemma}[Theorem]{Lemma}

\newtheorem{Remark}[Theorem]{Remark}
\newtheorem{Example}[Theorem]{Example}

\begin{document}

\title{Regularity Results for Eikonal-Type Equations \\with
Nonsmooth Coefficients}
\author{Piermarco Cannarsa\thanks{Dipartimento di Matematica, Universit\`a di Roma ``Tor Vergata'',
Via della Ricerca Scientifica 1, 00133 Roma (Italy),
e-mail: cannarsa@mat.uniroma2.it}\quad\&\quad Pierre Cardaliaguet \thanks{Universit\'e de Brest,
UMR 6205, 6 Av. Le Gorgeu,
BP 809, 29285 Brest (France); e-mail:
Pierre.Cardaliaguet@univ-brest.fr}}

\maketitle

\begin{abstract} 
Solutions of the Hamilton-Jacobi equation $H(x,-Du(x))=1$, with $H(\cdot,p)$ H\"{o}lder continuous and $H(x,\cdot)$ convex and positively homogeneous of degree $1$, are shown to be locally semiconcave with a power-like modulus. An essential step of the proof is the ${\mathcal C}^{1,\alpha}$-regularity of the extremal trajectories associated with the multifunction generated by $D_pH$.
\end{abstract} 



\section{Introduction}
The importance of semiconcavity for the study of Hamilton-Jacobi equations and optimal control problems is by now  widely acknowledged. Indeed, such a qualitative property ensures the upper semicontinuity and quasi-monotonicity of the superdifferential, provides upper bounds for the set where differentiability fails providing, at the same time, criteria for the propagation of singularities, and leads to stronger optimality conditions than the ones holding for a continuous (or Lipschitz continuous) function, see, for instance, \cite{casi04} and the references therein. 

Typically, a real-valued function $u$ is semiconcave on the convex set $D\subset\R^N$ if there exists a modulus  (i.e., nondecreasing upper semicontinuous function, vanishing at $0$) $\omega :[0,\infty)\to [0,\infty)$  such that
$$
u(\lambda x+(1-\lambda)y) \geq \lambda u(x)+(1-\lambda)u(y) - C\lambda(1-\lambda) |x-y|\omega(|x-y|)
$$
for all $x,y\in D$ and $\lambda\in[0,1]$.

Semiconcavity results with a linear modulus hold for viscosity solutions of Hamilton-Jacobi equations with convex Hamiltonians which are sufficiently smooth with respect to the space variables, as well as value functions of optimal control problems with smooth dynamics and running cost (see, e.g., \cite{li82}, \cite{is84}, \cite{cafr91}; see also \cite{casi04}). Known generalizations allow for Lipschitz continuous dependance with respect to space, provided the Hamiltonian is strictly convex and superlinear in the gradient variables (see \cite{caso89}, \cite{si95}).

In this paper we shall study the Dirichlet problem
\be\label{HJ:intro}
\left\{\begin{array}{ll}
H(x,-Du(x))=1 & {\rm in }\; \Omega\\
u(x)=0 & {\rm on}\; \partial \Omega
\end{array}\right.
\ee
where $\Omega$ is an open subset of $\R^N$, $H(x,\cdot)$ is  convex and positively homogeneous of degree 1, and 
$H(\cdot,p)$ is just H\"{o}lder continuous. Consequently, \eqref{HJ:intro} fits none of the aforementioned settings.
Nevertheless, our main result---Theorem~\ref{theo:main} below---guarantees that the solution $u$ of \eqref{HJ:intro} 
is locally semiconcave in $\Omega$ with the power-like modulus $\omega(t)=Ct^\theta$, for some  $\theta>0$ depending on $H$.

The method of proof relies on the representation of $u(x)$ as the minimum time needed to reach $\partial\Omega$ along a trajectory of the differential inclusion
\begin{equation}\label{DI:intro}
\begin{cases}
x'(t)\in F(x(t))
& t\ge 0\quad \text{a.e.}
\\
x(0)=x\,,
\end{cases}
\end{equation}
where
$$
F(x)={\rm co}\,\{D_pH(x,p)~:~ p\in \R^N\backslash \{0\}\}\qquad\forall x\in\R^N\,.
$$
An essential step of the analysis is the ${\mathcal C}^{1,\alpha}$-regularity of the extremal trajectories of \eqref{DI:intro}, see 
Theorem~\ref{ReguExtremal}.  For  time-dependent and isotropic Hamiltonians ($H=a(t,x)|p|$), such a regularity property---interesting in its own right---has already been observed  in \cite{subu07} for $N=2$, and \cite{calemo} for general $N$.  However, the unexpected connection between Theorem~\ref{ReguExtremal} and the  semiconcavity of the solution of \eqref{HJ:intro} is, to our best knowledge, entirely new. 

The main technical tools we borrow from convex analysis  are recalled in detail in section~\ref{se:preliminaries}, which makes  this paper essentially self-contained.

\section{Notation and assumptions}\label{se:notation}
Let $N$ be a positive integer. Denote by $\langle\cdot,\cdot\rangle$ and $|\cdot|$ the Euclidean scalar product and norm in $\R^N$, respectively, and set 
\begin{equation*}
\cb = \{x\in\R^N~:~|x|\le 1\}\,.
\end{equation*}
More generally, for all $x\in\R^N$ and $\rho>0$, $B(x,\rho)$ stands for the closed ball of radius $\rho$ centered at $x$, that is,
$B(x,\rho)=x+\rho B$.

Let $H:\R^N\times \R^N\to \R$ be a continuous function satisfying the following assumptions for some positive constants $C_0,r, R$, with $r<R$, and $\alpha\in (0,1/2)$.

\smallskip
\noindent
{\bf Standing Assumptions (SA):}
\begin{itemize}
\item
For all $p\in\R^N$, the function $x\mapsto H(x,p)$ is $2\alpha$-H\"{o}lder continuous, and 
\be\label{HHold}
\left| H(x,p)-H(y,p)\right|\leq C_0|x-y|^{2\alpha}|p|\qquad \forall x,y\in \R^N\,.
\ee
\item For all $x\in\R^N$, the function $p\mapsto H(x,p)$ is convex on $\R^N$, positively homogeneous of degree one, and
has linear growth, i.e.,
\be\label{croissanceH}
r|p|\ \leq\  H(x,p)\ \leq R|p|\ \qquad \forall p\in \R^N\;.
\ee
\item For all $x\in\R^N$, the function $p\mapsto H(x,p)$ is continuously differentiable on $\R^N\backslash \{0\}$, and, for all $p,q\in \R^N\backslash \{0\}$,
\be\label{ReguH}
\begin{array}{l}
\ds{ -\frac{1}{2r}\left|D_pH(x,q)-D_pH(x,p)\right|^2 } \\
\qquad \qquad \leq \; \ds{ \lg D_pH(x,q)-D_pH(x,p),\frac{p}{|p|}\rg  }\\
\qquad \qquad \qquad \qquad \ds{ \leq -\frac{1}{2R}\left|D_pH(x,q)-D_pH(x,p)\right|^2 }\,,
\end{array}
\ee
where $D_pH(x,p)$ denotes the gradient of $H$ in the $p$-variables at $(x,p)$.
\end{itemize}

Hereafter, by a universal constant---briefly, a constant---we mean a {\em positive} real number that only depends on the parameters $N,\alpha,r,R,$ and $C_0$ introduced above. 
Generic constants appearing in computations will be denoted by $C$. A subscript ($C_1, C_2,\dots$) will be added when necessary for future reference.
\begin{Example}\label{ex:Hamiltonian}\rm
 Let $H(x,p)=|A(x)p|$, where $A:\R^N\to \R^{N\times N}$ is $\alpha $-H\"{o}lder continuous
on $\R^N$, 
$A(x)$ is invertible for all $x\in\R^N$, and
$$
|A(x)|\leq C \quad {\rm and }\quad |A(x)^{-1}|\leq C \qquad \forall x\in \R^N
$$
for some constant $C$. Then $H$ satisfies (SA) for a suitable choice of constants.
\end{Example}

\section{Preliminary results}\label{se:preliminaries}
Let $H:\R^N\times \R^N\to \R$ be a continuous function satisfying our Standing Assumptions with fixed constants $\alpha,r,R,$ and $C_0$. 

For all $p\neq 0$, set $f_p(x)=D_pH(x,p)$ and define 
\begin{equation}\label{eq:F}
F(x)={\rm co}\,\{f_p(x)~:~ p\in \R^N\backslash \{0\}\}\qquad\forall x\in\R^N\,,
\end{equation}
where `co' stands for convex hull. Note that, for all $(x,p)\in \R^N\times (\R^N\backslash \{0\})$,
\begin{equation}\label{eq:H}
H(x,p)=\max_{v\in F(x)}\lg v,p\rg \qquad {\rm and }\qquad f_p(x)={\rm argmax}_{v\in F(x)} \lg v,p\rg\;.
\end{equation}

We begin by recovering some properties of $F(\cdot)$  that follow directly from (SA).
\begin{Lemma} The set-valued map $F$ is $2\alpha$-H\"{o}lder continuous, i.e.,
\be\label{HypHol}
F(x)\subset F(y)+C_0|x-y|^{2\alpha} B\qquad\forall x,y\in\R^N\,,
\ee
and satisfies, for all $x\in\R^N$, the {\em curvature} estimates 
\be\label{CondBoule}
B\Big(f_p(x)-r\frac{p}{|p|},r\Big)\subset F(x) \quad {\rm and}\quad F(x) \subset B\Big(f_p(x)-R\frac{p}{|p|},R\Big)\qquad\forall p\neq0\,,
\ee
as well as the {\em controllability} condition 
\be\label{control}
B(0,r)\ \subset \ F(x) \ \subset \ B(0,R)\,.
\ee\label{controlH}
\end{Lemma}

\begin{Remark}\label{re:inclusions}
{\rm  The first inclusion in  (\ref{CondBoule})---which can be interpreted as  an upper bound for the curvature of $\partial F(x)$---is equivalent to the inequality
$$
\Big|v-f_p(x)+r\frac{p}{|p|}\Big|\geq r\qquad \forall v\in \partial F(x)\,,
$$
which in turn can be recast as follows
\be\label{CondBouleIntEqDef}
- \frac1{2r} \left|v-f_p(x)\right|^2  \leq \Big\lg v-f_p(x), \frac{p}{|p|}\Big\rg \qquad \forall v\in \partial F(x)\,.
\ee
The second inclusion in (\ref{CondBoule})---a lower bound for the curvature---can be rephrased as
$$
\Big|v-f_p(x)+R\frac{p}{|p|}\Big|\leq R \qquad \forall v\in F(x)\,,
$$
which is equivalent to
\begin{equation*}
\Big\lg v-f_p(x),\frac{p}{|p|}\Big\rg \leq -\frac{1}{2R}|v-f_p(x)|^2\qquad \forall v\in F(x)\,,
\end{equation*}
or, since $F(x)$ is convex, 
\be\label{CondBouleExtEqDef}
\Big\lg v-f_p(x),\frac{p}{|p|}\Big\rg \leq -\frac{1}{2R}|v-f_p(x)|^2\qquad \forall v\in \partial F(x)\,.
\ee
}\end{Remark}
\begin{proof}
 Note that, since  $H(x,\cdot)$ is the support function of $F(x)$, inequality (\ref{HHold}) directly implies (\ref{HypHol})
while (\ref{croissanceH}) entails (\ref{control}). Let us now check that the regularity condition (\ref{ReguH}) implies (\ref{CondBoule}). 
For this we just have to note that, for any $v\in \partial F(x)$, there is some $q\neq 0$ such that $v=f_q(x)$, so that
(\ref{ReguH}) becomes
$$
-\frac{1}{r}\left|v-f_p(x)\right|^2
\leq 
\Big\lg v-f_p(x),\frac{p}{|p|}\Big\rg 
\leq -\frac{1}{R}\left|v-f_p(x)\right|^2 \qquad \forall v\in \partial F(x)\,.
$$
These two equalities are equivalent to (\ref{CondBouleIntEqDef}) and (\ref{CondBouleExtEqDef}). 
\end{proof}

Next, we derive a regularity result for $f_p(\cdot)$, which is actually a consequence of the H\"{o}lder continuity  of $F$ in (\ref{HypHol}) combined with the lower curvature bound in (\ref{ReguH}).

\begin{Lemma}\label{fpxfpy} For all $p\in \R^N\backslash\{0\}$ we have
$$
\left|f_p(x)-f_p(y)\right|\leq (C_0+\sqrt{2C_0R})\ |x-y|^\alpha \qquad \forall x,y\in \R^N\,.
$$
\end{Lemma}
\begin{proof}
 Let $x,y\in \R^N$. In view of (\ref{HypHol}) there are $v_y\in F(x)$ and $v_x\in F(y)$ such that 
\begin{equation*}
|f_p(y)-v_y|\leq C_0|x-y|^{2\alpha}\quad\mbox{and}\quad |f_p(x)-v_x|\leq C_0|x-y|^{2\alpha}\,.
\end{equation*}
Then, by (\ref{CondBouleExtEqDef}), 
$$
\Big\lg v_y-f_p(x),\frac{p}{|p|}\Big\rg \leq -\frac{1}{2R} |v_y-f_p(x)|^2 \; {\rm and }\; 
\Big\lg v_x-f_p(y),\frac{p}{|p|}\Big\rg \leq -\frac{1}{2R} |v_x-f_p(y)|^2 \;.
$$
Adding up the above two inequalities yields
$$
 |v_y-f_p(x)|^2+|v_x-f_p(y)|^2\ \leq -2R\  \Big\lg v_y-f_p(x)+v_x-f_p(y),\frac{p}{|p|}\Big\rg \;  \leq\;  2C_0R|x-y|^{2\alpha}\,.
 $$
 So,
 $$
 |f_p(y)-f_p(x)|\leq  |f_p(y)-v_y|+|v_y-f_p(x)|\leq (C_0+\sqrt{2C_0R})\ |x-y|^\alpha\,,
 $$
and the proof is complete.
\end{proof}

Describing the way how $f_p(x)$ depends on $p$ is the object of our next result.
\begin{Lemma}\label{lem:HYP2} For every $x\in \R^N$ we have
$$
\frac{1}{R} |f_p(x)-f_q(x)|\leq \left|\frac{p}{|p|}-\frac{q}{|q|}\right|\leq \frac{1}{r} |f_p(x)-f_q(x)|\,,\qquad\forall p,q\in \R^N\backslash \{0\}\,.
$$
\end{Lemma}
\begin{proof}
 Let $x\in \R^N$ and let $p,q\in \R^N\backslash \{0\}$. Let us start with the first inequality. Recalling the second condition in
(\ref{CondBoule}) in its equivalent form (\ref{CondBouleExtEqDef}), we have, since $f_q(x)\in \partial F(x)$,
$$
\Big\lg f_q(x)-f_p(x),\frac{p}{|p|}\Big\rg \leq -\frac{1}{2R}|f_q(x)-f_p(x)|^2\,.
$$
In a symmetric way we also have 
$$
\Big\lg f_p(x)-f_q(x),\frac{q}{|q|}\Big\rg \leq -\frac{1}{2R}|f_p(x)-f_q(x)|^2\,.
$$
Adding the two inequalities easily gives the first inequality.

We now prove the second inequality, which is slightly more subtle.  
 Recalling the first inclusion in (\ref{CondBoule}) and the definition of $f_p(x)$, we conclude that 
$$
\Big\lg q, f_p(x)-r\frac{p}{|p|}+rb\Big\rg \leq \lg q,f_q(x)\rg \,, \qquad\forall b\in B\,.
$$
Hence,
$$
-r \Big\lg \frac{q}{|q|}, \frac{p}{|p|}\Big\rg +r \leq \Big\lg f_q(x)-f_p(x),\frac{q}{|q|} \Big\rg \,.
$$
Thus, exchanging $p$ and $q$,
$$
-r \Big\lg \frac{p}{|p|}, \frac{q}{|q|} \Big\rg +r  \leq \Big\lg f_p(x)-f_q(x),\frac{p}{|p|} \Big\rg \,.
$$
Adding the above inequalities together leads to
\begin{equation}\label{eq:intermediate}
r \Big |\frac{p}{|p|}-\frac{q}{|q|} \Big |^2 
=
2r \Big(1-\Big\lg \frac{p}{|p|},\frac{q}{|q|} \Big\rg \Big) \leq \Big\lg f_p(x)-f_q(x), \frac{p}{|p|}- \frac{q}{|q|} \Big\rg\,.
\end{equation}
Since $f_p(x)$ and $f_q(x)$ are boundary points, 
\eqref{CondBouleIntEqDef} yields
$$
\Big\lg f_p(x)-f_q(x), \frac{p}{|p|}\Big\rg \leq \frac1{2r} \left|f_q(x)-f_p(x)\right|^2\,.
$$
and
$$
\Big\lg f_q(x)-f_p(x), \frac{q}{|q|}\Big\rg \leq \frac1{2r} \left|f_q(x)-f_p(x)\right|^2\,.
$$
Therefore,
\begin{equation}\label{eq:final}
\Big\lg f_p(x)-f_q(x), \frac{p}{|p|}-\frac{q}{|q|}\Big\rg \leq \frac{1}{r} |f_p(x)-f_q(x)|^2\;.
\end{equation}
The conclusion follows from \eqref{eq:intermediate} and \eqref{eq:final}.
\end{proof}

Let us now consider the {\em polar} of $H$, namely  the function $H^0$ defined by
$$
H^0(x,q):=\max\big\{\lg p,q\rg ~:~ H(x,p)\leq 1\big\}\qquad\forall (x,q)\in  \R^N\times \R^N\,.
$$
It is well-known that, for all $(x,q)\in \R^N\times \R^N$,
\be\label{H0leq1}
H^0(x,q)\leq 1 \qquad\Longleftrightarrow\qquad q\in F(x)\,,
\ee
and
\begin{equation}\label{eq:H^0}
H^0\big(x,D_pH(x,p)\big)=H^0\big(x,f_p(x)\big)=1 \qquad \forall (x,p)\in \R^N\times (\R^N\backslash \{0\})\,.
\end{equation}
The duality between $H$ and $H^0$ brings similar qualitative properties for these two functions. For instance, on account of (\ref{control}), we have
\be\label{control2}
\frac{|q|}{R}\leq H^0(x,q)\leq \frac{|q|}{r} \qquad \forall (x,q)\in \R^N\times \R^N\;.
\ee
Moreover, $H^0$ is also H\"{o}lder continuous with respect to $x$, with the same exponent as $H$.
\begin{Lemma}\label{le:H0}
 For all $q \in \R^N$,
 \be\label{ReguH0}
|H^0(x,q)-H^0(y,q)|\leq \frac{C_0}{r^2}|q||x-y|^{2\alpha} \qquad \forall x,y \in \R^N\;.
\ee
\end{Lemma}
\begin{proof}
 Let $x,y,q \in \R^N$. Take $p\in\R^N$, with $H(x,p)\leq 1$, such that $H^0(x,q)=\lg p,q\rg$. Then, by (\ref{croissanceH}), $|p|\leq 1/r$. Also, by (\ref{HHold}),
 \begin{equation*}
H(y,p)\leq 1+\frac{C_0}{r}|x-y|^{2\alpha}\,.
\end{equation*}
So,
$$
H^0(y,q)\geq \Big\lg\frac{p}{1+\frac{C_0}{r}|x-y|^{2\alpha}}, q\Big\rg =\frac{H^0(x,q)}{1+\frac{C_0}{r}|x-y|^{2\alpha}} \,.
$$
On the other hand, in view of \eqref{control2},
\begin{eqnarray*}
H^0(x,q)&=& \frac{H^0(x,q)}{1+\frac{C_0}{r}|x-y|^{2\alpha}} + 
\frac{\frac{C_0}{r}|x-y|^{2\alpha}}{1+\frac{C_0}{r}|x-y|^{2\alpha}}H^0(x,q)
\\
&\leq& \frac{H^0(x,q)}{1+\frac{C_0}{r}|x-y|^{2\alpha}} + \frac{C_0}{r^2}|q||x-y|^{2\alpha}\,.
\end{eqnarray*}
Thus,
\begin{equation*}
H^0(y,q)\geq H^0(x,q)- \frac{C_0}{r^2}|q||x-y|^{2\alpha}\,.
\end{equation*}
Hence, we obtain the conclusion exchanging the roles of $x$ and $y$. 
\end{proof}

We now turn to the analysis of the level set  
\begin{equation*}
F^0(x)= \left\{ p\in \R^N~:~H(x,p)\leq 1\right\}\qquad x\in\R^N\,.
\end{equation*}
\begin{Lemma}\label{lem:p/|p|vsp} Let $x\in\R^N$. Then, for every $p,p'\in \R^N$ with $H(x,p)=H(x,p')=1$,  
\begin{equation}\label{eq:le36}
|p-p'| \leq C\Big|\frac{p'}{|p'|}-\frac{p}{|p|}\Big|
\end{equation}
for some constant $C$.
\end{Lemma}
\begin{proof}
 
First of all, the reader be warned that, as $x$ plays no role in this proof,  the $x-$dependence in $H$  will be omitted. 
For all $\theta,\theta'\in  S^{N-1}$, we have
$$
\Big|\frac{\theta}{H(\theta)}-\frac{\theta'}{H(\theta')} \Big |
\leq \frac{|\theta-\theta'|}{H(\theta)} +\frac{|H(\theta)-H(\theta')|}{H(\theta)H(\theta')}\,.
$$
Since $H$ is Lipschitz continuous by \eqref{control}, recalling $r\leq H(\theta),H(\theta')\leq R$ we conclude that
$$
\frac{|\theta-\theta'|}{H(\theta)} +\frac{|H(\theta)-H(\theta')|}{H(\theta)H(\theta')}
\leq C|\theta-\theta'|
$$
for some constant $C$. Therefore, 
$$
\Big|\frac{\theta}{H(\theta)}-\frac{\theta'}{H(\theta')} \Big|
\leq C|\theta-\theta'|\;.
$$
Now, observe that the map $\theta\mapsto \theta/H(\theta)$ is a bijection between the unit sphere $S^{N-1}$ and $\partial F^0(x)$. 
So, applying the above inequality to $\theta,\theta'\in S^{N-1}$ chosen such that $p=\theta/H(\theta)$ and
$p'=\theta'/H(\theta')$ we obtain the conclusion.
\end{proof}
 
\begin{Lemma}[Lower curvature estimate for $F^0$] \label{lem:LowerCurvF0}
There is a constant $R'$ such that
$F^0(x)$ satisfies the lower curvature estimate of radius $R'$ for all $x\in\R^N$, i.e.,
$$
F^0(x)\subset B\Big( \frac{p}{H(x,p)} -R'\frac{f_p(x)}{|f_p(x)|},R'\Big)\qquad \forall x,p\in \R^N, \;p\neq 0
$$
or, equivalently, 
\begin{equation}\label{eq:equivbe}
\Big\lg p'-p, \frac{f_p}{|f_p|} \Big\rg \leq -\frac{1}{2R'}|p'-p|^2 \qquad \forall p,p'\in \partial F^0
\end{equation}
\end{Lemma}
\begin{proof}
 Again,  we shall drop the $x$-dependence in all the formulas below since it is of no interest for this proof. 
Recalling Remark~\ref{re:inclusions} we conclude that it suffices to 
prove inequality (\ref{eq:equivbe}) for some constant $R'$.  
Let then $p,p'\in \partial F^0$. Since $H$ is  positively homogeneous of degree 1,  we have
$$
\begin{array}{rl}
H(p')-H(p)- \lg D_pH(p), p'-p\rg \; = & 
\lg D_pH(p'),p'\rg -\lg D_pH(p),p\rg - \lg D_pH(p), p'-p\rg \\
= & \lg D_pH(p')-D_pH(p),p'\rg
\end{array}
$$
where $D_pH(p)=f_p$ and $D_pH(p')=f_{p'}$. From the lower curvature estimate on $F$ given in (\ref{CondBouleExtEqDef}) it follows that
$$
\Big\lg f_{p}-f_{p'}, \frac{p'}{|p'|} \Big\rg \leq -\frac{1}{2R}\left|f_{p'}-f_p\right|^2\;.
$$
Thus, combining the above inequality with the previous identity, and using the fact that $H(p)=H(p')=1$, 
$$
\lg f_p, p'-p\rg \;\leq \;- \frac{|p'|}{2R}\left|f_{p'}-f_p\right|^2\;.
$$
Now, apply Lemma~\ref{lem:HYP2}  to obtain 
\begin{equation}\label{eq:le37}
\lg f_p, p'-p\rg \;\leq \; -\frac{r^2|p'|}{2R} \Big|\frac{p'}{|p'|}-\frac{p}{|p|} \Big|^2\,.
\end{equation}
Finally, let $C$ be the constant given by Lemma~\ref{lem:p/|p|vsp}. Then, \eqref{eq:le36} and \eqref{eq:le37} yield 
$$
\Big\lg \frac{f_p}{|f_p|}, p'-p \Big\rg \;\leq \; -\,\frac{r^2|p'|}{2R^2}\Big|\frac{p'}{|p'|}-\frac{p}{|p|} \Big|^2\; \leq \; -\,\frac{r^2}{2C^2R^3}\left|p'-p\right|^2\,.
$$
whence the conclusion follows  with $R'= C^2R^3/r^2$.
\end{proof}

In particular, Lemma~\ref{lem:LowerCurvF0} ensures $F^0(x)$ is a strictly convex set for any $x\in \R^N$. Thus, since $H^0(x,\cdot)$ is the support function of $F^0(x)$, $D_qH^0(x,q)$ exists for any $x,q\in \R^N$ with $q\neq 0$
(see, for instance, \cite[Theorem~A.1.20]{casi04}). In fact, we shall soon prove a stronger property: the map $q\to D_qH^0(x,q)$ is locally Lipschitz continuous in $\R^N\backslash \{0\}$. Before doing this, let us collect  some technical remarks on the link between $H^0$ and $H$ and their derivatives.

\begin{Lemma}\label{lem:AnalH0} We have, for any $p,q\in \R^N\backslash \{0\}$, 
\be\label{qFDq=pF0Dp}
\left[q\in \partial F(x)\;{\rm and }\; p=D_qH^0(x,q)\right]\qquad 
\Leftrightarrow \qquad
\left[p\in \partial F^0(x)\;{\rm and }\; q=D_pH(x,p)\right]
\ee
In particular, 
\begin{equation}\label{DqH0}
D_qH^0\Big(x,\frac{f_p(x)}{|f_p(x)|} \Big)= p \qquad \forall p\in \partial F^0(x)\,.
\end{equation}
\end{Lemma}

\begin{proof} We just need to show the implication 
$$
\left[q\in \partial F(x)\;{\rm and }\; p=D_qH^0(x,q)\right]\qquad \Rightarrow \qquad
\left[p\in \partial F^0(x)\;{\rm and }\; q=D_pH(x,p)\right]
$$
because $H^{00}=H$. Let $q\in \partial F(x)$ and $p=D_qH^0(x,q)$. Note that $H(x,p)=H^0(x,q)=1$ and, in particular, $p\in \partial F^0(x)$. By definition, we have
\be\label{H0H}
H^0(x,q')H(x,p')\geq \lg p', q'\rg \qquad \forall p', q'\in \R^N\,.
\ee
This inequality becomes an equality for $(p',q')=(p,q)$ because 
$$
\lg p, q\rg= \lg D_qH^0(x,q), q\rg = H^0(x,q)=1= H^0(x,q)H(x,p)\;.
$$
Taking the derivative in (\ref{H0H}) with respect to $p$ then gives 
$$
H^0(x,q)D_pH(x,p)=D_pH(x,p)= q\;.
$$

Next we turn to the proof of (\ref{DqH0}). Recall first that $f_p(x)= D_pH(x,p)$ for any $p\neq 0$. So, if $p\in \partial F^0(x)$, then (\ref{qFDq=pF0Dp}) implies that 
$$
D_qH^0\Big(x,\frac{f_p(x)}{|f_p(x)|} \Big)= D_qH^0\Big(x,f_p(x) \Big)= p\;,
$$
since $D_qH^0(x,\cdot)$ is $0-$homogeneous. 
 \end{proof}

\begin{Lemma}\label{lem:lipDqH0} There is a constant $C$ such that, for every $x\in\R^N$,
$$
\left|D_qH^0(x,q)-D_qH^0(x,q')\right|\leq \frac{C}{|q|\vee |q'|} |q-q'|\qquad\forall q,q'\in \R^N\setminus\{0\}\,.
$$
\end{Lemma}
\begin{proof}
Let us fix $p,p'\in \partial F^0(x)$. Owing to Lemma~\ref{lem:LowerCurvF0} in its equivalent form \eqref{eq:equivbe}, we deduce that
$$
\Big\lg p'-p, \frac{f_p(x)}{|f_p(x)|} \Big\rg \leq -\frac{1}{2R'} \left| p'-p\right|^2
$$
and
$$
\lg p-p', \frac{f_{p'}(x)}{|f_{p'}(x)|}\rg \leq -\frac{1}{2R'} \left| p'-p\right|^2\;.
$$
Adding up the last two inequalities,
\begin{equation}\label{p-p'}
\left| p'-p\right|^2 \leq R'\Big\lg p'-p, \frac{f_{p'}(x)}{|f_{p'}(x)|}-\frac{f_{p}(x)}{|f_{p}(x)|} \Big\rg \leq R'
|p'-p| \Big |\frac{f_{p'}(x)}{|f_{p'}(x)|}-\frac{f_{p}(x)}{|f_{p}(x)|} \Big |\,.
\end{equation}
Now, recall that the map $p\mapsto f_p(x)/|f_p(x)|$ is a bijection from $\partial F^0(x)$ to $S^{N-1}$ to deduce that for all  $q,q'\in S^{N-1}$ there are $p,p'\in \partial F^0(x)$ such that $q= f_p(x)/|f_p(x)|$ and $q'=f_{p'}(x)/|f_{p'}(x)|$. Then, combining (\ref{DqH0}) and (\ref{p-p'}), 
$$
\begin{array}{rl}
\left|D_qH^0\left(x,q'\right)-D_qH^0\left(x,q\right)\right|\; = & 
\Big|D_qH^0 \Big(x,\frac{f_{p'}(x)}{|f_{p'}(x)|} \Big)-D_qH^0 \Big(x,\frac{f_p(x)}{|f_p(x)|} \Big) \Big |  \\
 = &  |p'-p|\vspace{0.2cm} \\
 \leq & R'\left|\frac{f_{p'}(x)}{|f_{p'}(x)|}-\frac{f_{p}(x)}{|f_{p}(x)|}\right| =  R'|q'-q|\,.
\end{array}
$$
This is the desired estimate for 
$q,q'\in S^{N-1}$. Next, let $q,q'\in \R^N\backslash\{0\}$. Then, since  $D_qH^0(x,\cdot)$ is homogeneous of degree $0$,  
$$
\left|D_qH^0(x,q')-D_qH^0(x,q)\right|\leq R' \Big|\frac{q'}{|q'|}-\frac{q}{|q|} \Big|\,.
$$
Finally, observe that
$$
\Big|\frac{q'}{|q'|}-\frac{q}{|q|} \Big| \leq \Big|\frac{q'}{|q'|}-\frac{q}{|q'|} \Big|+ \Big|\frac{q}{|q'|}-\frac{q}{|q|} \Big|
= \frac{|q'-q| }{|q'|}+ \frac{|q|}{|q||q'|} |\ |q'|-|q|\ | \leq 2 \frac{|q'-q| }{|q'|}
$$
to complete the proof. \end{proof}

\section{Regularity of extremal trajectories}\label{se:extremal}
In this section, we shall prove a regularity result for the extremal trajectories of the differential inclusion 
\be\label{InclDiff}
x'(t)\in F(x(t)) \qquad t\geq 0\,,
\ee
where $F$ is the multifunction introduced in \eqref{eq:F}, and $H$ is a given function satisfying (SA). Alternatively, this analysis could be addressed to differential inclusions associated with a multifunction $F:\R^N \rightrightarrows\R^N$ that satisfies \eqref{HypHol}, \eqref{CondBoule}, and \eqref{control} as standing assumptions, in which case the Hamiltonian $H$ should be defined as in \eqref{eq:H}. 

A {\em trajectory} of the above differential inclusion is a locally absolutely continuous arc $x(\cdot):[0,\infty)\to\R^N$ that satisfies  \eqref{InclDiff} for a.e. $t\ge 0$. Given a closed subset $K$ of $\R^N$, we denote by ${\mathcal R}(t)$, $t\ge 0$, the {\em reachable set} (from $K$) in time $t$, that is,
$$
{\mathcal R}(t)=\left\{x(t)~:~x(\cdot)\; \mbox{\rm is a trajectory of (\ref{InclDiff}) with $x(0)\in K$}\right\}\,.
$$
A trajectory $\bar x(\cdot)$ of (\ref{InclDiff}) is called extremal on the time interval $[0,t]$ if $\bar x(t)\in \partial {\mathcal R}(t)$.
In this case, one can show that in fact $\bar x(s)\in {\mathcal R}(s)$ for every $s\in [0,t]$. 

Due to the special structure of $F$, described by the properties \eqref{HypHol}, \eqref{CondBoule}, and \eqref{control}, we will be able to show that all extremal trajectories are ${\mathcal C}^{1,\alpha/2}$-smooth. More precisely, we have the following result.
\begin{Theorem}\label{ReguExtremal}
Assume {\rm (SA)} and let  $\bar x$ be an extremal trajectory of  \eqref{InclDiff} on some time interval $[0,T]$.Then 
\be\label{estix'}
|\bar x'(t_2)-\bar x'(t_1)|\leq C(t_2-t_1)^{\alpha/2} \qquad \forall t_1,t_2\in [0,T]
\ee
for some constant $C$.
\end{Theorem}
\begin{proof}
 Let $\bar x$ be an extremal trajectory on $[0,T]$. Then, by extremality, $x'(t)\in \partial F(\bar x(t))$ for almost all $t\in [0,T]$, so that we can set 
 $$
 \bar p(t)= D_qH^0(\bar x(t), \bar x'(t)) \qquad \mbox{\rm a.e. in }[0,T]\;.
 $$
Using Lemma \ref{lem:AnalH0}, we obtain the following relation between $\bar x$ and $\bar p$:
$$
\bar x'(t)= D_pH\big(\bar x(t),\bar p(t)\big) \quad \mbox{ for a.e.}\quad t\in [0,T]\,.
$$
{\bf Step 1.} We first claim that, for any $0\leq t_1<t_2\leq T$ we have 
\be\label{klaim1} 
t_2-t_1 \; \leq \;  H^0\Big(\bar x(t_2), \int_{t_1}^{t_2} D_pH\big(\bar x(t_2),\bar p(t)\big)dt\Big) +C(t_2-t_1)^{1+\alpha} \,.
\ee
{\it Proof of (\ref{klaim1}): } Let us set
$$
q= \frac{\bar x(t_2)-\bar x(t_1)}{|\bar x(t_2)-\bar x(t_1)|}\;.
$$
Let $\lambda:[t_1,t_2]\to\R$ be a solution of the Cauchy problem
$$
\begin{cases}
\ds\lambda'(t)=\frac{1}{H^0(\bar x(t_1)+\lambda(t)q,q)}\,,
& t\in [t_1,t_2]
\\
\lambda(t_1)=0\,.
\end{cases}
$$
Then $x(t):=\bar x(t_1)+\lambda(t) q$ is a trajectory of (\ref{InclDiff}) since,  owing to (\ref{H0leq1}), 
$$
H^0\big(x(t), x'(t) \big)= H^0 \big(x(t), \lambda'(t)q \big) = \frac{H^0 \big(\bar x(t_1)+\lambda(t) q, q \big)}{ H^0 \big(\bar x(t_1)+\lambda(t)q,q \big)}=1
$$
for all $t\in [t_1,t_2]$. Therefore, since $\bar x$ is an extremal trajectory, the point $\bar x(t_1)+\lambda(t_2)q$
belongs to the segment $[\bar x(t_1), \bar x(t_2)]$. So,
$$
\lambda(t_2)-\lambda(t_1) \leq |\bar x(t_2)-\bar x(t_1)|\;.
$$
Note that, owing to (\ref{ReguH0}), the above inequality, and the boundedness of $F$,
$$
\Big|\frac{1}{H^0(\bar x(t_1)+\lambda(t)q,q)}-\frac{1}{H^0(\bar x(t_2),q)}\Big|\le C|\bar x(t_2)-\bar x(t_1)|^{2\alpha}
\le C(t_2-t_1)^{2\alpha}
$$
for all $t\in [t_1,t_2]$ and some constant $C$. Hence, 
$$
\lambda(t_2)-\lambda(t_1) = \int_{t_1}^{t_2} \frac{dt}{H^0(\bar x(t_1)+\lambda(t)q,q)} \geq \frac{t_2-t_1}{H^0(\bar x(t_2),q)}
-C(t_2-t_1)^{1+2\alpha}\,.
$$
So, appealing to Lemma \ref{fpxfpy},
$$
\begin{array}{rl}
t_2-t_1 \; \leq & H^0(\bar x(t_2),q)|\bar x(t_2)-\bar x(t_1)|+C(t_2-t_1)^{1+2\alpha} 
\vspace{.2cm}\\
= & H^0(\bar x(t_2),\bar x(t_2)-\bar x(t_1))+C(t_2-t_1)^{1+2\alpha} \vspace{.2cm}\\
= & H^0\Big(\bar x(t_2), \int_{t_1}^{t_2} D_pH(\bar x(t),\bar p(t))dt\Big) +C(t_2-t_1)^{1+2\alpha} \vspace{.2cm}\\
\leq & H^0 \Big(\bar x(t_2), \int_{t_1}^{t_2} D_pH(\bar x(t_2),\bar p(t))dt \Big) +C(t_2-t_1)^{1+\alpha} \,,
\end{array}
$$
where the above constants may change from line to line. We have thus proved (\ref{klaim1}). \\

\noindent
{\bf Step 2.} Let us fix $0\leq t_1<t_2\leq T$ and let $\bar t$ be such that 
\be\label{ChoixBarT}
H^0(\bar x(t_2), \bar x(\bar t)-\bar x(t_1))= H^0(\bar x(t_2), \bar x(t_2)-\bar x(\bar t))\;.
\ee
Define
\begin{equation}\label{eq:defab}
a= \bar x(\bar t)-\bar x(t_1) \quad {\rm and }\quad b= \bar x(t_2)-\bar x(\bar t)\,.
\end{equation}
We claim that
\be\label{klaim2}
H^0(\bar x(t_2),a)+H^0(\bar x(t_2),b) \leq H^0(\bar x(t_2),a+b)+ C(t_2-t_1)^{1+\alpha}
\ee
and 
\be\label{klaim3}
|2\bar t-t_1-t_2|\leq C(t_2-t_1)^{1+\alpha}\,.
\ee

\noindent {\it Proof of (\ref{klaim2}) and (\ref{klaim3}): } Again by Lemma \ref{fpxfpy}, and then using Jensen's inequality, we obtain
\begin{equation}\label{eq:stima1}
\begin{array}{rl}
H^0(\bar x(t_2), a) \; = & H^0\Big(\bar x(t_2), \int_{t_1}^{\bar t} D_pH (\bar x(s), \bar p(s))ds \Big) \vspace{.2cm}
\\
\leq & H^0 \Big(\bar x(t_2), \int_{t_1}^{\bar t} D_pH (\bar x(t_2), \bar p(s))ds \Big) + C(t_2-t_1)^{1+\alpha}
\vspace{.2cm}\\
\leq & \int_{t_1}^{\bar t}  H^0 \Big(\bar x(t_2), D_pH (\bar x(t_2), \bar p(s)) \Big)ds + C(t_2-t_1)^{1+\alpha}
\vspace{.2cm}\\
\leq & \bar t-t_1+ C(t_2-t_1)^{1+\alpha}\,.
\end{array}
\end{equation}
Applying (\ref{klaim1}) between $t_1$ and $\bar t$ gives 
\begin{equation}\label{eq:prima}
\bar t-t_1 \; \leq   H^0 \Big(\bar x(\bar t), \int_{t_1}^{\bar t} D_pH(\bar x(\bar t),\bar p(t))dt \Big) +C(\bar t-t_1)^{1+\alpha} \,.
\end{equation}
Now, in order to bound the above right-hand side observe that
\begin{equation*}
\big|D_pH(\bar x(\bar t),\bar p(t))-D_pH(\bar x(t),\bar p(t)) \big|
\le C|\bar x(\bar t)-\bar x(t)|^{\alpha}\le C(\bar t-t_1)^{\alpha} 
\end{equation*}
in view of Lemma~\ref{fpxfpy}, and
\begin{equation*}
\big|H^0 \big(\bar x(\bar t), a \big)-H^0 \big(\bar x(t_2), a \big)\big|
\le C\,|a|\,|\bar x(\bar t)-\bar x(t_2)|^{2\alpha}\le C (\bar t_2-t_1)^{1+2\alpha} 
\end{equation*}
owing to Lemma~\ref{le:H0}. Therefore, \eqref{eq:prima} leads to
\begin{equation}\label{eq:stima2}
\begin{array}{rl}
\bar t-t_1 \; \leq &  H^0 \Big(\bar x(\bar t), \underbrace{\int_{t_1}^{\bar t} D_pH(\bar x(t),\bar p(t))dt}_{a} \Big) +C(\bar t-t_1)^{1+\alpha} 
\vspace{.2cm}\\
\leq &  H^0(\bar x(t_2), a) +C(t_2-t_1)^{1+\alpha}\,.
\end{array}
\end{equation}
On account of \eqref{eq:stima1} and \eqref{eq:stima2}, we have
$$
H^0(\bar x(t_2), a) - C(t_2-t_1)^{1+\alpha} \leq \bar t-t_1 \leq H^0(\bar x(t_2), a) + C(t_2-t_1)^{1+\alpha}\;.
$$
In the same way,
$$
H^0(\bar x(t_2), b ) - C(t_2-t_1)^{1+\alpha} \leq t_2-\bar t \leq H^0(\bar x(t_2), b) + C(t_2-t_1)^{1+\alpha}\;.
$$
Combining the above two inequalities with the choice of $\bar t$ made in (\ref{ChoixBarT}) gives (\ref{klaim3}). 
Moreover, adding up the above inequalities and recalling (\ref{klaim1}), we get
$$
H^0(\bar x(t_2), a)+H^0(\bar x(t_2), b ) \leq (t_2-t_1)+ C(t_2-t_1)^{1+\alpha}
 \leq H^0(\bar x(t_2), a+b)+ C(t_2-t_1)^{1+\alpha}\,,
 $$
 which yields (\ref{klaim2}). \\
 
\noindent
{\bf Step 3.} We now claim that, for any $0\leq t_1<t_2\leq T$, we have
\be\label{klaim4}
\Big|\bar x\Big(\frac{t_1+t_2}{2}\Big)-\frac{\bar x(t_2)+\bar x(t_1)}{2}\Big|\leq C(t_2-t_1)^{1+\alpha}\,.
\ee
{\it Proof of (\ref{klaim4}): } Having fixed $0\leq t_1<t_2\leq T$, we will use the same notation for $\bar t$, $a$, and $b$ as in  (\ref{ChoixBarT}) and \eqref{eq:defab}.
Moreover, since $x(t_2)$ is fixed in the reasoning below, as we often did before we will omit the $x(t_2)$-dependance of $H^0$
and all the other maps appearing in this proof.

Let us set, for any $q\in \R^N\backslash \{0\}$, $g_q=D_qH^0(q)$. We use below repetitively the following remark:  
$$
\mbox{\rm for any $q\neq 0$, if $p=g_q$, then $f_p=q/H^0(q)$.}
$$
Indeed, since $q/H^0(q)\in \partial F$ and $p=D_qH^0(q/H^0(q))$ (because $D_qH^0$ is $0-$homogeneous), Lemma \ref{lem:AnalH0} implies that $q/H^0(q)= D_pH(p)=f_p$. 

We first show that 
\be\label{eq:gqgqprim}
\frac{1}{2R'} |g_{q}-g_{q'}|^2\leq \lg g_{q}-g_{q'}, \frac{q}{|q|}\rg \qquad \forall q,q'\in \R^N\backslash \{0\}\;,
\ee
where $R'$ is the constant appearing in Lemma \ref{lem:LowerCurvF0}. For this, let us consider the lower curvature estimate (\ref{eq:equivbe}) in Lemma \ref{lem:LowerCurvF0} with $p=g_q$ and $p'=g_{q'}$: because of the remark above and since $p,p'\in \partial F^0$, we have
$$
\Big\lg g_{q'}-g_q, \frac{q}{|q|} \Big\rg \leq -\frac{1}{2R'}|g_{q'}-g_q|^2
$$ 
which is exactly (\ref{eq:gqgqprim}). 

Next we note that 
\be\label{eq:a-b}
|a-b| \leq CH^0(a) |g_a-g_b|
\ee
Indeed let us apply the first inequality in Lemma \ref{lem:HYP2} to $p=g_a$ and $q=g_b$. Since $f_p=a/H^0(a)$ and $f_q=b/H^0(b)$ and since $H^0(a)=H^0(b)$ by (\ref{ChoixBarT}), we have
$$
|a-b| \leq R H^0(a) \left| \frac{g_a}{|g_a|}-\frac{g_b}{|g_b|}\right|\leq C H^0(a) |g_a-g_b|\;.
$$
In order to estimate the right-hand side of inequality (\ref{eq:a-b}), let us observe that, in view of (\ref{klaim2}), 
$$
\begin{array}{rl}
0\;  \leq & H^0(a+b)-H^0(a)-H^0(b)+C(t_2-t_1)^{1+\alpha} \\
= & \lg g_{a+b},a+b\rg-\lg g_a,a\rg-\lg g_b,b\rg +C(t_2-t_1)^{1+\alpha}
\end{array}
$$
so that
$$
0\leq \lg g_{a+b}-g_a,a\rg+\lg g_{a+b}-g_b,b\rg +C(t_2-t_1)^{1+\alpha}\;.
$$
Plugging inequality  (\ref{eq:gqgqprim}) into this inequality leads to
$$
|a|\left|g_{a+b}-g_a\right|^2 +|b|\left|g_{a+b}-g_b\right|^2 \leq C(t_2-t_1)^{1+\alpha}\;,
$$
i.e., since $|a|\geq (\bar t-t_1)/C$ and $|b|\geq (t_2-\bar t)/C$ and (\ref{klaim3}) holds, 
$$
\left|g_{a+b}-g_a\right| \leq C (t_2-t_1)^{\alpha/2}\;, \; 
 \left|g_{a+b}-g_b\right| \leq C (t_2-t_1)^{\alpha/2}\;.
$$
So, recalling that $H^0(a) \leq C(\bar t-t_1)\leq C (t_2-t_1)$, we get from (\ref{eq:a-b}),
$$
|a-b| \; \leq \; CH^0(a)(|g_{a+b}-g_a(x)|+|g_{a+b}-g_b(x)|)\  \leq\ C(t_2-t_1)^{1+\alpha/2}\;.
$$
From the definition of $a$ and $b$, this means that 
$$
\left|2\bar x(\bar t)-\bar x(t_1)-\bar x(t_2)\right|\leq C(t_2-t_1)^{1+\alpha/2}\;.
$$
Using again (\ref{klaim3}) and the Lipschitz continuity of $\bar x$ then easily yields to (\ref{klaim4}). \\

\noindent
{\bf Conclusion.} In view of (\ref{klaim4}), Theorem 2.1.10 of \cite{casi04} states that each component of $\bar x$ is semi-convex and 
semi-concave with a modulus $m$ of the form $m(\rho)=C\rho^{\alpha/2}$. Then, from Theorem 3.3.7 of \cite{casi04},
$\bar x$ is of class ${\mathcal C}^{1,\alpha/2}$ and (\ref{estix'}) holds.
\end{proof}

\section{The semiconcavity result}\label{se:semiconcavity}
Let $H:\R^N\times \R^N\to \R$ be a continuous function satisfying our Standing Assumptions with constants $\alpha,r,R,$ and $C_0$, and let  $\Omega\subset\R^N$ be an open set. 

In this section, we will apply the previous analysis to study the regularity of the solution to the Dirichlet problem
\be\label{HJ}
\left\{\begin{array}{ll}
H(x,-Du(x))=1 & {\rm in }\; \Omega\\
u(x)=0 & {\rm on}\; \partial \Omega
\end{array}\right.
\ee
The existence, uniqueness, and Lipschitz continuity of the viscosity solution $u$ of the above problem is well-known, as well as the representation formula
\begin{equation}\label{eq:mintime}
u(x)=\inf\big\{t\geq 0~:~\exists \; x(\cdot) \;\mbox{\rm trajectory of (\ref{InclDiff}) with $x(0)=x\,,\;x(t)\in \partial \Omega$}\big\}
\end{equation}
(see, e.g., \cite{bardi1997optimal}).

We recall that a function $v:\Omega\to\R$ is locally $\theta$-semiconcave, with $\theta\in (0,1]$, if for every compact convex set ${\mathcal O}\subset \Omega$ there is a constant $C_{\mathcal O}$ such that 
$$
v(\lambda x+(1-\lambda)y) \geq \lambda v(x)+(1-\lambda)v(y) - C_{\mathcal O}\lambda(1-\lambda) |x-y|^{1+ \theta}$$
for all $x,y\in {\mathcal O}$ and $\lambda\in[0,1]$. We are now ready for our main result.
\begin{Theorem}\label{theo:main}
Assume {\rm (SH)}. Then the solution $u$ of (\ref{HJ}) is locally $\theta $-semiconcave in $\Omega$ for every $\theta\in (0,\frac{\alpha}{4+\alpha})$. \end{Theorem} 
\begin{proof}
The strategy of the proof is the following. Fix
\begin{equation}
\label{eq:defbeta}
\beta\in \Big(\frac{2}{2+\alpha},\frac{4}{4+\alpha}\Big)\,,
\end{equation}
and observe that
\begin{equation*}
0<\theta:=
\beta\,\frac{2+\alpha}2-1< \frac{\alpha}{4+\alpha}\,.
\end{equation*}
Let ${\mathcal O}\subset\subset \Omega$ be an open convex set.  We are going to show that
\begin{equation}\label{eq:conclusion}
u(x+h)+u(x-h)-2u(x)\; \leq \; C |h|^{\beta(2+\alpha)/2}
\end{equation}
for  all $h\in \R^N$ sufficiently small (in this proof, $C$ denotes a generic constant depending only on $\alpha,r,R,C_0$, and ${\mathcal O}$).  Since $u$ is continuous, owing to  \cite[Theorem 2.1.10]{casi04} the above inequality implies that $u$ is locally $\theta$-semiconcave in $\Omega$.\\

\noindent{\bf Step 1.}
Let  $\bar x\in {\mathcal O}$ and let $\bar x(\cdot)$ be a solution of the minimization problem in \eqref{eq:mintime}---an optimal trajectory for short. Since $\bar x(\cdot)$ is
extremal on $[0,u(\bar x)]$, Theorem \ref{ReguExtremal} implies that $\bar x(\cdot)$ is of class ${\mathcal C}^{1,\alpha/2}$
and satisfies
$$
\left|\bar x'(t_2)-\bar x'(t_1)\right|\leq C|t_2-t_1|^{\alpha/2}\qquad \forall t_1,t_2\in [0, u(\bar x)]\,.
$$
Setting $\bar v=\bar x'(0)$, from the above inequality we obtain
\be\label{keyestim}
\left|\bar x(t)-\bar x- t\bar v\right|\leq C t^{1+\alpha/2} \qquad \forall t\in [0, u(\bar x)]\,.
\ee

\noindent{\bf Step 2.} Let $h\in \R^N$ be small enough, and set $\bar t=|h|^{\beta}$. We will now build a trajectory $x_+(\cdot)$  such that 
\begin{equation*}
\begin{cases}
x_+(0)=\bar x+h\,,
&
\\
x_+(t)\in[\bar x+h,\bar x(\bar t)] &\forall t\in [0,\tau_+]\,,
\\
x_+(t)=\bar x(t+\bar t- \tau_+) & \forall t\in [\tau_+, u(\bar x)+\tau_+-\bar t]\,.
\end{cases}
\end{equation*}
Notice that
$x_+(u(\bar x)+\tau_+-\bar t)=\bar x(u(\bar x))\in \R^N\backslash \Omega$, so that
\begin{equation}\label{eq:sc+}
u(\bar x+h) \leq 
u(\bar x)+\tau_+-\bar t\,.
\end{equation}
{\it Proof of Step 2.} In order to construct the line segment part, let us set $q_+=\bar x(\bar t)-(\bar x+h)$
 and observe that, in view of (\ref{keyestim}),
\begin{equation}\label{eq:estiq+}
q_+=\bar t \bar v+O(\bar t^{1+\alpha/2})-h\,.
\end{equation}
Then,
$q_+\neq 0$  since
$\bar t>>|h|$ and $|\bar v|\geq 1/r$. 
Let $\lambda(\cdot)$ be a solution of the Cauchy problem
$$
\begin{cases}
\ds\lambda'(t)=\frac{1}{H^0(\bar x+h+\lambda(t)q_+, q_+)}\,,
& t\ge 0
\\
\lambda(0)=0\,.
\end{cases}
$$
Since $\lambda(\cdot)$ is strictly increasing there is a unique time $\tau_+$ such that $\lambda (\tau_+)=1$. 
Now, set 
$$x_+(t)=\bar x+h+\lambda(t)q_+\qquad t\in[0,\tau_+]\,.
$$ 
Then $x_+(\cdot)$ is a solution of the differential inclusion (\ref{InclDiff})
on $[0,\tau_+]$ because 
$$H^0(x_+(t),x_+'(t))= 1\qquad \forall t\in [0,\tau_+]\,.$$ 
Moreover $x_+(\tau_+)= \bar x(\bar t)$. Thus, defining 
$$x_+(t)=\bar x(t+\bar t- \tau_+)\qquad\forall t\in [\tau_+, u(\bar x)+\tau_+-\bar t]$$
completes the construction of $x_+(\cdot)$.\\

\noindent{\bf Step 3.}
We will now  prove the estimate  
\begin{equation}\label{eq:stimatau}
\tau_+\leq \bar t+ \big\lg D_qH^0(\bar x,\bar t\bar v), q_+-\bar t\bar v\big\rg +
C\,\frac{|h|^2}{\bar t}\,.
\end{equation}
{\it Proof of Step 3.} To begin with, let us note that any $\xi\in [\bar t\bar v,q_+]$ satisfies $|\xi|\geq \bar t/C$. Indeed, if $\xi=\mu q_++(1-\mu)\bar t\bar v$
for some $\mu\in [0,1]$, then, by (\ref{eq:estiq+}),  $\xi= \bar t\bar v+O(\bar t^{1+\alpha/2})-h$ with
$|h|=\bar t^{1/\beta}$ and $\beta\in (0,1)$. So, for $|h|$ small enough, we have the desired claim: $|\xi|\geq \bar t/C$.
Then, by Lemma \ref{lem:lipDqH0} we conclude that the map $\xi\mapsto D_qH^0(\bar x, \xi)$ is Lipschitz on $[\bar t\bar v,q_+]$ with constant $C/\bar t$. So, for all $t\in [0,\tau_+]$,
\begin{multline*}
H^0(x_+(t), q_+) \; \leq   H^0(\bar x, q_+)+ C|x_+(t)-\bar x|^{2\alpha} |q_+| 
\\
\leq  H^0(\bar x, \bar t\bar v)+ \lg D_qH^0(\bar x,\bar t\bar v), q_+-\bar t\bar v\rg +
(C/\bar t)|q_+-\bar t\bar v|^2+
C|x_+(t)-\bar x|^{2\alpha} |q_+|\,.
\end{multline*}
Since $x_+(t)\in[\bar x+h, \bar x(\bar t)]$,  we have
$$
|x_+(t)-\bar x|\leq \max \{ |h|, |\bar x(\bar t)-\bar x|\} \leq C\bar t\;.
$$
Also, on account of (\ref{eq:estiq+}) and \eqref{eq:defbeta}, 
$$
|q_+|\leq C\bar t\qquad {\rm and }\qquad |q_+-\bar t\bar v|= |O(\bar t^{1+\alpha/2})-h|\leq C |h|\,.
$$
Noting that $H^0(\bar x, \bar v)=1$ because $\bar v\in \partial F(\bar x)$, the above inequality yield, by the homogeneity of $H^0(\bar x,\cdot)$,
$$
H^0(x_+(t), q_+) \; \leq \; \bar t+ \big\lg D_qH^0(\bar x,\bar t\bar v), q_+-\bar t\bar v\big\rg +
C\Big( \frac{|h|^2}{\bar t}+
\bar t^{1+2\alpha}\Big)\,,
$$
where $|h|^2/\bar t>\bar t^{1+2\alpha}$ because $\bar t=|h|^\beta$ with $\beta>1/(1+\alpha/2)$. 
So,
$$
H^0(x_+(t), q_+) \; \leq \; \bar t+ \big\lg D_qH^0(\bar x,\bar t\bar v), q_+-\bar t\bar v \big\rg +
C|h|^2/\bar t\,.
$$
Then 
$$
1=\int_0^{\tau_+}\lambda'(t)dt=\int_0^{\tau_+}\frac{dt}{H^0(x_+(t),q_+)}
 \geq \frac{\tau_+}{\bar t+ \big\lg D_qH^0(\bar x,\bar t\bar v), q_+-\bar t\bar v \big\rg +
C|h|^2/\bar t}\,,
$$
which in turn yields \eqref{eq:stimatau}.\\

\noindent{\bf Conclusion.}
Repeating the above reasoning with $q_-=\bar x(\bar t)-(\bar x-h)$, we can build  a solution $x_-(\cdot)$ to (\ref{InclDiff}) such that $x_-(0)=\bar x-h$, 
$x_-(t)\in[\bar x-h,\bar x(\bar t)]$ on the time interval $[0,\tau_-]$, and
 $x_-(t)=\bar x(t+\bar t- \tau_-)$ on $[\tau_-, u(\bar x)+\tau_--\bar t]$. Therefore,
\begin{equation}\label{eq:sc-}
u(\bar x-h) \leq 
u(\bar x)+\tau_--\bar t\,,
\end{equation}
where $\tau_-$ can be estimated as above: 
\begin{equation}\label{eq:stimatau-}
\tau_-\leq 
\bar t+ \big\lg D_qH^0(\bar x,\bar t\bar v), q_--\bar t\bar v \big\rg +
C\,\frac{|h|^2}{\bar t}\,.
\end{equation}
Hence, by \eqref{eq:sc+}, \eqref{eq:sc-}, \eqref{eq:stimatau}, \eqref{eq:stimatau-}, and \eqref{keyestim} we obtain 
$$
\begin{array}{rl}
u(\bar x+h)+u(\bar x-h)-2u(\bar x)\; \leq & \tau_+-\bar t+\tau_--\bar t \vspace{.2cm}\\
\leq & 2 \big\lg DH^0(\bar x,\bar t\bar v), \bar x(\bar t)-\bar x-\bar t\bar v \big\rg +
C\,\frac{|h|^2}{\bar t}
\vspace{.2cm}\\
\leq & C |h|^{\beta(2+\alpha)/2}
\end{array}
$$
since  $\beta<4/(4+\alpha)$. We have thus attained \eqref{eq:conclusion}, which completes the proof. 
\end{proof}


\begin{thebibliography}{1}

\bibitem{bardi1997optimal}
M.~Bardi and I.~Capuzzo~Dolcetta.
\newblock {\em {Optimal control and viscosity solutions of
  Hamilton-Jacobi-Bellman equations}}.
\newblock Springer, 1997.

\bibitem{subu07}
B.Su and M.Burger.
\newblock Global weak solutions of non-isothermal front propagation problem.
\newblock {\em Electron. Res. Announc. Amer. Math. Soc.}, 13:46--52, 2007.

\bibitem{cafr91}
P.~Cannarsa and H.~Frankowska.
\newblock Some characterizations of optimal trajectories in control theory.
\newblock {\em SIAM J. Control Optim.}, 29:1322--1347, 1991.

\bibitem{casi04}
P.~Cannarsa and C.~Sinestrari.
\newblock {\em Semiconcave functions, Hamilton-Jacobi equations and optimal
  control}.
\newblock Birkh\"{a}user, 2004.

\bibitem{caso89}
P.~Cannarsa and H.M. Soner.
\newblock Generalized one-sided estimates for solutions of hamilton-jacobi
  equations and applications.
\newblock {\em Nonlinear Anal.}, 13(3):305--323, 1989.

\bibitem{calemo}
P.~Cardaliaguet, O.Ley, and A.~Monteillet.
\newblock Viscosity solutions for a polymer crystal growth model.
\newblock {\em Indiana Univ. Math. J. (to appear)}.

\bibitem{is84}
H.~Ishii.
\newblock Uniqueness of unbounded viscosity solutions of hamilton-jacobi
  equations.
\newblock {\em Indiana Univ. Math. J.}, 33:721--748, 1984.

\bibitem{li82}
P.L. Lions.
\newblock {\em Generalized solutions of Hamilton-Jacobi equations}.
\newblock Pitman, 1982.

\bibitem{si95}
C.~Sinestrari.
\newblock Semiconcavity of solutions of stationary hamilton-jacobi equations.
\newblock {\em Nonlinear Anal.}, 24:1321--1326, 1995.

\end{thebibliography}

%
%
%
%

\end{document}